\documentclass[11pt,a4paper]{amsart}
\usepackage[latin1]{inputenc}
\usepackage{latexsym}
\usepackage{color,graphicx,shortvrb}
\usepackage{amsmath, amssymb}
\usepackage{amsfonts}
\usepackage[colorlinks, bookmarks=true]{hyperref}

\newtheorem{theorem}{Theorem}[section]

\newtheorem{corollary}[theorem]{Corollary}

\theoremstyle{definition}

\newtheorem{remark}[theorem]{Remark}

\author{J. M. Almira and Z. Boros}
\title{A dichotomy property for the graphs of monomials}

\thanks{Research of the second author is supported by
the Hungarian Scientific Research Fund (OTKA) grant K-111651.}
\begin{document}


\subjclass[2010]{}


\begin{abstract}
We prove that the graph of a  discontinuous $n$-monomial  function
$f:\mathbb{R}\to\mathbb{R}$ is either connected or totally disconnected.
Furthermore, the discontinuous monomial functions with connected graph
are cha\-rac\-terized as those satisfying a certain big graph property.
Finally, the connectedness properties of the graphs of additive functions
$f:\mathbb{R}^d\to\mathbb{R}$ are studied.
\end{abstract}

\maketitle

\markboth{J. M. Almira and Z. Boros}{Dichotomy property for monomials}

\section{Motivation}
F. B. Jones \cite{jones} proved in 1942, in a famous paper,
the existence of additive discontinuous functions
$f:\mathbb{R}\to\mathbb{R}$ whose graph
$G(f)=\{(x,f(x)):x\in\mathbb{R}\}$ is connected, and characterized them.
These functions  are extraordinary since
their graphs are dense connected subsets of the plane,
containing exactly one point in each vertical line $\{x\}\times\mathbb{R}$ \cite{sanjuan}.
In his paper the author also stated, without proof,  that
the graph of a discontinuous additive function must be connected or totally disconnected.
For this result he just referenced another famous paper, by Hamel \cite{hamel},
but the proof is not there.
Indeed, up to our knowledge, a proof of this dichotomy result
has never appeared in the literature.
In this note we prove that the graph of a discontinuous monomial is
either connected or totally disconnected,
and we characterize the discontinuous monomial functions
$f:\mathbb{R}\to\mathbb{R}$ with connected graph
by means of a big graph property.
We also study the connected components of the graphs of additive functions
$f:\mathbb{R}^d\to\mathbb{R}$ for $d\geq 1$.
These results should be a good starting point to prove that,
for larger classes of functions,
such as the generalized polynomials or exponential polynomials over the real line,
the graphs of the elements of these sets are either connected or totally disconnected.
To find examples of both situations is an easy corollary of the structure of these functions
and Jones's existence result of additive discontinuous functions with connected graph.

\section{Dichotomy property for monomials}
Recall that $f:\mathbb{R}\to\mathbb{R}$ is an $n$-monomial function
if it is a solution of the so called monomial functional equation
\begin{equation}
\label{monomials}
\frac{1}{n!}\Delta^{n}_hf(x)=f(h) \ \ (x,h\in\mathbb{R}).
\end{equation}
It is known that $f$ satisfies $(\ref{monomials})$ if and only if
$f(x)=F(x,\cdots,x)$
for a certain multi-additive and symmetric function
$F:\mathbb{R}^n\to\mathbb{R}$,
and that $f$ is a polynomial function of degree at most $n$
(i.e., $f$ solves Fr\'{e}chet's functional equation $\Delta^{n+1}_hf(x)=0$)
if and only if $f(x)=\sum_{k=0}^nf_k(x)$,
where $f_k(x)$ is a $k$-monomial function for $k=0,1,\cdots, n$.
(See, for example, \cite{czerwik}, \cite{kuczma}, for the proofs of these claims).

\begin{theorem}
[Dichotomy, for monomial functions $f:\mathbb{R}\to\mathbb{R}$]
\label{main1}
Let $f:\mathbb{R}\to\mathbb{R}$ be a $n$-monomial function.
Then $G(f)$ is either connected or totally disconnected.
Furthermore, both cases are attained by concrete examples of discontinuous $n$-monomials
$f:\mathbb{R}\to\mathbb{R}$, for every $ n \in \mathbb{N} $.
\end{theorem}

\begin{proof}
Let $f$ be an $n$-monomial. Suppose that $G(f)$ is not totally disconnected.
Then there exists a connected component $H\subset G(f)$
containing at least two different points $(x_1,y_1)$ and $(x_2,y_2)$.
Clearly, $x_1\neq x_2$.
Hence, if $\pi_1:\mathbb{R}^2\to\mathbb{R}$ denotes the horizontal projection of the plane,
$\pi_1(x,y)=x$, then
$I=\pi_1(H)=\{x\in\mathbb{R}: (x,f(x))\in H\}$
is a connected subset of $\mathbb{R}$ which contains two distinct points.
Hence $I$ is an interval with non-empty interior and $H=\{(x,f(x)):x\in I\}$.
Set $\alpha=\inf I$ and $\beta=\sup I$.
Obviously, $\alpha,\beta\in \mathbb{R}\cup\{-\infty,+\infty\}$ and $\alpha<\beta$.
In particular either $\beta>0$ or $\alpha<0$.
Assume, with no loss of generality, that $\beta>0$
(the other case can be treated with similar arguments, or reduced to this one
by using that every monomial is either even or odd,
since $f(rx)=r^nf(x)$ for all rational number $r$ and all $x\in\mathbb{R}$,
which implies that $f(-x)=(-1)^n f(x)$ for all $x\in\mathbb{R}$).

Given $q\in\mathbb{Q}\setminus\{0\}$ we consider the maps
\[
\phi_{q,k}(x,y)=(q^kx,q^{kn}y); \ \text{ whenever  }
(x,y)\in\mathbb{R}^2 \text{ and } k\in\mathbb{Z}.
\]
Then $\phi_{q,k}:\mathbb{R}^2\to\mathbb{R}^2$ is continuous.
Hence $H_{q,k}=\phi_{q,k}(H)$ is a connected subset of the plane for every $k\in\mathbb{Z}$.
Given $k\in\mathbb{Z}$, we have that $H_{q,k}\cap H_{q,k+1}\neq \emptyset$
if and only if  for some $(x,f(x)), (x^*,f(x^*))\in H$
the equality $\phi_{q,k}(x,f(x))=\phi_{q,k+1}(x^*,f(x^*))$ holds.
In other words, this intersection is nonempty if and only if
\[
(q^kx,q^{kn}f(x))=(q^{k+1}x^*,q^{(k+1)n}f(x^*)) \text{ for certain } x,x^*\in I.
\]
Forcing equality between the first components of these vectors we get $q^{k}x=q^{k+1}x^*$,
which means that $x^*=\frac{1}{q}x$.
Furthermore, under this restriction,
we get the equality between the second components of the vectors for free, since
\[
q^{(k+1)n}f(x^*)= q^{(k+1)n}f \left( \frac{1}{q}x \right)
= q^{(k+1)n} \left(\frac{1}{q}\right)^nf(x) = q^{kn}f(x).
\]
Thus, we have demonstrated that $H_{q,k}\cap H_{q,k+1}\neq \emptyset$
if and only if there exist $x\in \mathbb{R}$ such that $\{x,\frac{1}{q} x\}\subset  I$.
In particular, when this holds true,
the property is satisfied for all $k\in\mathbb{Z}$ simultaneously and
$\widetilde{H}_q=\bigcup_{k\in\mathbb{Z}}H_{q,k}$ is connected.
Furthermore, $(t,f(t))\in G(f)$ implies that
$\phi_{q,k}(t,f(t))=(q^kt,(q^{k})^nf(t))= (q^kt,f(q^kt))\in G(f)$,
so that $H_{q,k}=\phi_{q,k}(H)\subseteq G(f)$ for all $k\in\mathbb{Z}$.
Hence $\widetilde{H}_q$ is always a subset of $G(f)$.

Assume, by the moment, that $0<\alpha$ and $\beta<\infty$,
and take $q\in\mathbb{Q}$ such that $1<q<\frac{\beta}{\alpha}$.
Then  $0<\alpha$ and  $1<q<\frac{\beta}{\alpha}$ imply that $\alpha<q\alpha <\beta$.
Take $x$ such that $q\alpha < x<\beta$.
Then $x\in I$ and $\alpha <\frac{1}{q}x<\frac{1}{q}\beta<\beta$,
so that $\frac{1}{q}x\in I$ too.
Thus, in this case, $\widetilde{H}_q$ is a connected subset of $G(f)$.
But we also have, in this case, that  $\widetilde{H}_q=G_{+}^*(f):= \{(x,f(x)): x>0\}$.
This implies that $(0,\infty)\subseteq I$,
which contradicts both $\alpha >0$ and $\beta <+\infty$.
Hence either $ \alpha \leq 0 $ or $\beta=+\infty$.

If $0<\alpha <\beta=+\infty$ and $0<x^*\in I$, then  for any $q>1$,
$x=qx^*$ and $\frac{1}{q}x=x^*$ both belong to $I$,
so that $\widetilde{H}_q= G_{+}^*(f)$ is connected, which contradicts $0<\alpha$.

If $ \alpha \leq 0 $ and $ 0 < x \in I $, then  for any $q>1$,
$x^*=\frac{1}{q}x$ and $x=qx^*$ both belong to $I$,
so that $\widetilde{H}_q$ is connected and contains $G_{+}^*(f)$.
This forces $\beta=+\infty$ again.
If, in particular, $ \alpha < 0 $, we have $ 0 \in I $, and we obtain,
analogously to the previous arguments, that $(-\infty,0)\subseteq I$.
In this case we thus have $ I = \mathbb{R} $ and $ H = G(f) $.

Finally, let us consider the case $ \alpha = 0 $.
Then either $ I = [0,\infty) $ or $ I = (0,\infty) $.
In the former case
$ H = G_{+}(f) := \{(x,f(x)): x\geq 0\} $ is connected.
Furthermore,  if we define $\varphi(x,y)=(-x,(-1)^n y)$, it is clear that
$ G_{-}(f) := \{(x,f(x)): x\leq 0\} = \varphi(G_{+}(f)) $
is also a connected subset of $G(f)$.
Furthermore, $ (0,0) = (0,f(0)) \in G_{+}(f) \cap G_{-}(f) $, so that
$ G(f) = G_{+}(f) \cup G_{-}(f) $ is connected.
In the latter case, when $ I = (0,\infty) $,
we have that $ G_{+}^*(f) $ is connected and
$ G_{+}(f) = G_{+}^*(f) \cup \{(0,0)\} $ is disconnected.
Hence there exist open sets $ U \subset \mathbb{R}^2 $ and $ V \subset \mathbb{R}^2 $
such that $ U \cap V = \emptyset $,
$ U \cap G_{+}(f) \neq \emptyset $, $ V \cap G_{+}(f) \neq \emptyset $,
and $ G_{+}(f) \subseteq U \cup V $.
We may assume, with no loss of generality, that $ (0,0) \in U $.
Then $ V \cap G_{+}^*(f) \neq \emptyset $.
Since
\[
(0,0) =
\lim_{m \to \infty} \left( \frac{1}{m} \,,\, \frac{1}{m^n} f(1) \right) =
\lim_{m \to \infty} \left( \frac{1}{m} \,,\, f \left( \frac{1}{m} \right) \right)
\]
is an accumulation point of $ G_{+}^*(f) \,$,
we obtain $ U \cap G_{+}^*(f) \neq \emptyset $ as well.
This yields that $ G_{+}^*(f) $ is disconnected, which is a contradiction.

Till now, we have demonstrated that,
if $f:\mathbb{R}\to\mathbb{R}$ is a monomial function,
then $G(f)$ is either connected or totally disconnected.
Let us now show that both cases are attained by concrete examples.
For totally disconnected graphs the example can be easily constructed.
Indeed, given  $\gamma$ any Hamel basis of $\mathbb{R}$ satisfying $1\in\gamma$
and let $n\in\mathbb{N}$ be a positive integer, we consider
$A_{\gamma}:\mathbb{R}\to\mathbb{R}$, the unique $\mathbb{Q}$-linear map
which satisfies $A_{\gamma}(1)=1$ and $A_{\gamma}(b)=0$
for every $b\in\gamma\setminus\{1\}$.
Obviously $f_n(x)=A_{\gamma}(x)^n$ is an $n$-monomial and
$f_{n}(\mathbb{R})\subseteq \mathbb{Q}$.
Hence the graph of $f_{n}$ is totally disconnected.

The  existence of discontinuous $n$-monomials with connected graph
follows from the existence of discontinuous additive functions
$f:\mathbb{R}\to\mathbb{R}$ with connected graph $G(f)$,
a fact that was demonstrated by Jones by using
a nontrivial set theoretical argument on ordinals \cite[Theorems 4 and 5]{jones}.
Indeed, assume that
$f:\mathbb{R}\to\mathbb{R}$ is additive, discontinuous, and $G(f)$ is connected.
Then $F(x)=x^{n-1}f(x)$ is a discontinuous $n$-monomial function with connected graph,
since the function $\phi:\mathbb{R}^2\to\mathbb{R}^2$ given by $\phi(x,y)=(x,x^{n-1}y)$
is continuous and transforms the graph of $f$ onto the graph of $F$.
\end{proof}

\begin{corollary}
[Dichotomy, for additive functions $f:\mathbb{R}\to\mathbb{R}$]
\label{dichotomy}
Let $f:\mathbb{R}\to\mathbb{R}$ be an additive function.
Then $G(f)$ is connected or totally disconnected.
Furthermore, there exists discontinuous additive functions
$f:\mathbb{R}\to\mathbb{R}$ with connected graph $G(f)$.
\end{corollary}

\noindent \textbf{Proof. }
Any additive function is a $1$-monomial function.
{\hfill $\Box$}

\begin{remark}
If $G(f)$ is connected, we have two cases:
either $f$ is continuous and $G(f)=V$ is a one-dimensional vector space,
or $G(f)$ is a connected dense additive subgroup of $\mathbb{R}^2$.
\end{remark}

The following theorem may be also of interest:


\begin{theorem}
[$(d+2)$-chotomy property of additive functions] \label{d}
If $f:\mathbb{R}^d\to\mathbb{R}$ is an additive function, then:
\begin{itemize}
\item[$(a)$]
There exists $s\in\{0,1,\cdots,d+1\}$ such that
the connected component $G$ of $G(f)$ which contains the point $(0,0)$
is a dense subgroup of an $s$-dimensional vector subspace of $\mathbb{R}^{d+1}$.
Furthermore, every connected component of $G(f)$ results from $G$ by a translation.
\item[$(b)$]
All cases described in $(a)$ are attained by concrete examples.
 \end{itemize}
\end{theorem}

\noindent \textbf{Proof. } $(a)$
Previous to introduce the main argument,
it is necessary to recall two basic facts about topological groups.
Concretely, if $G$ is a topological group, and $G_0$ denotes the identity component of $G$
(i.e., the biggest connected subset of $G$ which contains the identity $e\in G$),
then $G_0$  is a closed normal subgroup of $G$.
Furthermore, the elements of  the quotient group $G/G_0$ are
just the connected components of $G$ \cite{P}.
We consider $f:\mathbb{R}^d\to\mathbb{R}$ an additive function
and we set $G(f)=\{(x,f(x)):x\in \mathbb{R}^d\}$.  Obviously,
the additivity of $f$ implies that $G(f)$ is an additive subgroup of $\mathbb{R}^{d+1}$.
Let $G$ be the connected component of $G(f)$ which contains the zero element.
Then every connected component of $G(f)$ results from $G$ by a translation.

$G$ is a connected additive subgroup of $\mathbb{R}^{d+1}$.
Hence, its topological closure $\overline{G}$ is also a connected subgroup of $\mathbb{R}^{d+1}$.
It is known that the topological closure of any additive subgroup $H$ of $\mathbb{R}^{d+1}$
satisfies $\overline{H}=V\oplus \Lambda$ for a certain vector subspace $V$ of $\mathbb{R}^{d+1}$
and a discrete additive subgroup $\Lambda$ of $\mathbb{R}^{d+1}$
(see  \cite[Theorem 3.1]{Wald} for a proof of this fact).
It follows that $\overline{G}=V$ for a certain vector subspace $V$ of
$\mathbb{R}^{d+1}$. Hence every connected component of $G(f)$ is
the translation $\tau+G$ of a dense connected additive subgroup $G$
of the vector space $V$ for some $\tau\in \mathbb{R}^{d+1}$.
Note that, if $V=\{0\}$ then $G(f)$ is totally disconnected and, if $V=\mathbb{R}^{d+1}$,
then $G(f)$ is a connected dense additive subgroup of $\mathbb{R}^{d+1}$.
All the other cases represent an intermediate situation.
For example, if $f$ is continuous, then
$G(f)=V$ is a $d$-dimensional vector subspace of $\mathbb{R}^{d+1}$.

\noindent $(b)$
All cases described by Theorem \ref{d} can be constructed easily,
since all functions $f:\mathbb{R}^d\to\mathbb{R}$ of the form
$f(x_1,\cdots,x_d)=A_1(x_1)+A_2(x_2)+\cdots +A_d(x_d)$,
with $A_k:\mathbb{R}\to\mathbb{R}$ additive for each $k$, are additive,
and we can use the dichotomy result for each one of these functions $A_k$, $k=1,\cdots,d$.

{\hfill $\Box$}

\begin{remark}
While searching in the literature for a demonstration of Corollary \ref{dichotomy},
the first author commented this question to Professor L\'{a}szl\'{o} Sz\'{e}kelyhidi,
who also was unable to find the proof nowhere.
Then, he got a very nice independent proof of the result \cite{laszlo}.
Indeed, for $d=1$ we get the dichotomy result as follows (this is Sz\'{e}kelyhidi's idea):
Let  $\pi_1:\mathbb{R}\times \mathbb{R}\to \mathbb{R}$ denote the horizontal projection
$\pi_1(x,y)=x$,
and let $W=\pi_1(G)$ be the projection of the connected component $G$ of $G(f)$
which contains the zero element.
Then $\pi_1(G)=\{0\}$ or $\pi_1(G)=\mathbb{R}$,
since the only connected subgroups of the real line are $\{0\}$ and $\mathbb{R}$.
Thus, if $G(f)$ is not totally disconnected, then $\pi_1(G)=\mathbb{R}$,
which implies $G=G(f)$ and hence, $G(f)$ is connected.
Unfortunately, this simple proof seems to be very difficult to generalize
for the case of monomial functions $f:\mathbb{R}\to\mathbb{R}$,
since the graph of an $n$-monomial function is in general not an additive subgroup of $\mathbb{R}^2$.
We hope this justifies to introduce the proof of Theorem \ref{main1}.
\end{remark}

\section{A big-graph property}

Recently, Almira and Abu-Helaiel  characterized
the topological closures of the graphs of monomial functions as follows
\cite[Theorem 2.7]{AK_monomials}:

\begin{theorem}[Almira, Abu-Helaiel] \label{Graph_Monomials}
Assume that $f:\mathbb{R}\to\mathbb{R}$ is a discontinuous $n$-monomial function,
let $\Gamma_f=\overline{G(f)}^{\mathbb{R}^2}$,
and let us consider the function $A_n(h)=f(h)/h^n$, for $h\neq 0$.
Let $\alpha=\sup_{h\in\mathbb{R}\setminus\{0\}}A_n(h)$ and
$\beta=\inf_{h\in\mathbb{R}\setminus\{0\}}A_n(h)$.  Then:
\begin{itemize}
\item[$(a)$]
If $\alpha=+\infty$ and $\beta=-\infty$, then $\Gamma_f =\mathbb{R}^2$.
\item[$(b)$]
If $\alpha=+\infty$ and $\beta\in\mathbb{R}$, then
$\Gamma_f =\{(x,y):y\geq \beta x^{n}\}$ if $n=2k$ is an even number, and
$\Gamma_f =\{(x,y):x\leq 0 \text{ and } y\leq \beta x^{n}\} \cup
           \{(x,y):x\geq 0\text{ and } y\geq \beta x^{n}\}$
if $n=2k+1$ is an odd number.
In particular, if $\beta=0$, we get
the half space $\Gamma_f =\{(x,y):y\geq 0\}$ for $n=2k$ and
the union of the first and third quadrants $\Gamma_f =\{(x,y):xy\geq 0\}$, for $n=2k+1$.
\item[$(c)$]
If $\alpha\in\mathbb{R}$ and $\beta=-\infty$, then
$\Gamma_f =\{(x,y):y\leq \alpha x^{n}\}$ if $n=2k$ is an even number, and
$\Gamma_f =\{(x,y):x\leq 0 \text{ and } y \geq \alpha x^{n}\} \cup
           \{(x,y):x\geq 0\text{ and } y\leq \alpha x^{n}\}$
if $n=2k+1$ is an odd number.
In particular, if $\alpha=0$, we get
the half space $\Gamma_f =\{(x,y):y\leq 0\}$ for $n=2k$ and
the union of the second and fourth quadrants $\Gamma_f =\{(x,y):xy\leq 0\}$, for $n=2k+1$.
\end{itemize}
Furthermore, for all $n\geq 2$ there are examples of discontinuous $n$-monomial functions $f$
verifying each one of the claims $(a),(b),(c)$ above.
\end{theorem}

We use this result to prove the following big graph property:

\begin{theorem}[Big graph property] \label{main}
Let $f:\mathbb{R}\to\mathbb{R}$ be a discontinuous $n$-monomial function and let
$\Gamma_f=\overline{G(f)}^{\mathbb{R}^2}$ and $\Omega_f = \mathrm{Int} (\Gamma_f)$.
Then $G(f)$ is connected if and only if $G(f)$ intersects all continuum
$K\subseteq \Omega_f$ which touches two distinct vertical lines.
\end{theorem}
\begin{remark}
Recall that continuum means connected and compact with more than one point.
\end{remark}

\begin{proof}
The proof follows the very same arguments used by Jones \cite{jones} in his original proof for the case of additive functions. The main difference is that, for additive functions, the closure of the graph of a discontinuous additive function is the all plane and, for monomials, the corresponding sets are those shown in Theorem  \ref{Graph_Monomials}. 

Assume that $G(f)$ is not connected.
Then $G(f)\subseteq U\cup V$ with $U,V$ open subsets of the real plane,
$U\cap V=\emptyset$, $G(f)\cap U\neq \emptyset$ and $G(f)\cap V\neq \emptyset$.
We can assume that $U$ is connected,
since connected components of open subsets of $\mathbb{R}^2$ are open sets.
Indeed, by making $U$ or $V$ bigger and bigger,
just deleting properly some parts of the borders $\partial U$ or $\partial V$,
we can assume that both $U$ and $V$ are connected
and share a common border $\partial U=\partial V$.
Furthermore, this common frontier is necessarily connected,
since the connectedness of the boundary of an open domain in $\mathbb{R}^2$
is equivalent to the connectedness of its complement
(indeed, this result holds true for domains in $\mathbb{R}^n$ for all $n>1$ \cite{CKL}).

Now, the density of $G(f)$ in $\Gamma_f=\overline{\Omega_f}^{\mathbb{R}^2}$
implies that $V\cap \Gamma_f=\text{Ext}_{\Gamma_f}(U\cap \Gamma_f)$.
To prove this, we first observe that
$V\cap \Gamma_f\subseteq \text{Ext}_{\Gamma_f}(U\cap \Gamma_f)$,
since $V\cap \Gamma_f$ is an open set in the relative topology of $\Gamma_f$
which has empty intersection with $U\cap \Gamma_f$.
Thus, if  $V\cap \Gamma_f \neq \text{Ext}_{\Gamma_f}(U\cap \Gamma_f)$ ,
then there exist  $\varepsilon>0$ and $(x_0,y_0)\in \Gamma_f$ such that
$B((x_0,y_0),\varepsilon)\cap \Gamma_f \subseteq \text{Ext}_{ \Gamma_f }(U\cap \Gamma_f)\setminus V$,
which contradicts that $G(f)\subseteq U\cup V$,
since $G(f)$ has at least one point in $B((x_0,y_0),\varepsilon)\cap \Gamma_f$.

Now we can use the characterization of the sets $\Gamma_f$ given in Theorem \ref{Graph_Monomials}
to claim that  $\partial U \cap \Omega_f$ contains a continuum
which intersects two distinct vertical lines,
since otherwise $\partial U$ should contain the intersection of a vertical line with $\Gamma_f$,
a fact which leads to a contradiction, since $G(f)$ is a graph
and hence intersects all vertical lines.
This proves that, if $G(f)$ intersects
all continuum $K\subseteq \Omega_f$ which touches two distinct vertical lines,
then $G(f)$ is connected.

Let us now assume that $G(f)$ is connected and let $K\subseteq \Omega_f$ be a continuum
which touches two distinct vertical lines.
If $K$ has non-empty interior then $G(f)\cap K\neq \emptyset$, since $G(f)$ is dense in $\Gamma_f$.
If $\text{Int}(K)=\emptyset$ and $(x_0,y_0),(x_1,y_1)\in K$ with $x_0<x_1$, then,
$K\cap ([x_0,x_1]\times\mathbb{R})$ separates  $([x_0,x_1]\times\mathbb{R})\cap \Gamma_f$
in two (or more) components, since $K$ does not intersect the frontier of $\Gamma_f$.
Now, $G(f)$ contains at least a point of each one of these components, since $G(f)$ is dense in $\Gamma_f$.
It follows that $K\cap G(f)\neq \emptyset$, since  $G(f)$ is connected, by hypothesis.
Hence, if $G(f)$ is connected, then $G(f)$ intersects every continuum $K\subseteq \Omega_f$
which touches two distinct vertical lines.
\end{proof}

\begin{remark}
We can use the characterization above for another proof of the dichotomy property for monomials as follows:

If $f$ is continuous then $G(f)$ is connected.
Hence we assume that $f$ is a discontinuous $n$-monomial function.

As a first step, we reduce our study to the case of monomial functions with even degree,
by demonstrating that $G(f)$ is connected if and only if $G(g)$ is connected, where
$g(x)=xf(x)$.

The implication $G(f)$ connected implies $G(g)$ connected is trivial.
Let us prove the other implication.
Indeed, assume that $G(g)$ is connected with $g(x)=xf(x)$, $f(x)$ a $(2k+1)$-monomial function.
Let $K$ be a continuum included into $\Omega_f$ which touches two distinct vertical lines.
Then $F=\{(x,xy):(x,y)\in K\}$ is a continuum,  $F\subseteq \Omega_g$,
and $F$  touches two distinct vertical lines.
Hence Theorem \ref{main} and the connectedness of $G(g)$ imply that
there exists $x_0\neq 0$ such that $(x_0,g(x_0))=(x_0,x_0f(x_0))\in F$.
Thus $(x_0,f(x_0))\in K$ and $G(f)$ contains a point of $K$.
It follows, again from Theorem \ref{main}, that $G(f)$ is connected.

Let us thus assume (with no loss of generality) that $n=2k$ is even.
Thanks to Theorem \ref{Graph_Monomials} we can also assume with no loss of generality that
$\Gamma_f =\mathbb{R}^2$ or
$\Gamma_f =\{(x,y):y\geq \beta x^{2k}\}$ for a certain $\beta\in\mathbb{R}$,
since the other cases have analogous proofs.

If $G(f)$ is not connected,
there exist a continuum $K\subseteq \Omega_f$ with empty interior
and two points $(x_0,y_0),(x_1,y_1)\in K$ with $x_0<x_1$,
such that $G(f)\cap K=\emptyset$.
Obviously, the continuum $K$ separates $(]x_0,x_1[\times \mathbb{R}) \cap \Gamma_f$
in several disjoint  open subsets of $\Gamma_f$ (with the relative topology).
Hence we can assume that
$$(]x_0,x_1[\times \mathbb{R})\cap\Gamma_f \setminus K=U_K\cup V_K,$$
with $U_K,V_K$ disjoint open subsets of $\Gamma_f$, $U_K$ connected, and
$\{\alpha\}\times [\beta,+\infty)\subseteq U_K$ for certain $\alpha\in ]x_0,x_1[$ and $\beta>0$.

Let us set $U_K^*=U_K\cup \{(x,y)\in \partial U_K: (x,y)\not\in\partial V_K\}$,
$V_K^*=V_K\cup \{(x,y)\in \partial V: (x,y)\not\in\partial U_K\}$.
Then $U_K^*,V_K^*$ are open connected subsets of $\Gamma_f$,
$U_K^*\cap V_K^*=\emptyset$,
$G(f)\cap ]x_0,x_1[\times \mathbb{R} \subseteq U_K^*\cup V_K^*$,
$K^*=\partial U_K^*=\partial V_K^*$ is a continuum
which separates $ ]x_0,x_1[\times \mathbb{R} \cap \Gamma_f$
in exactly two disjoint open connected subsets of $\Gamma_f$,
$U_{K^*}=U_K^*$ and $V_{K^*}=V_K^*$.
$G(f)\cap U_{K^*}\neq \emptyset$ and $G(f)\cap V_{K^*}\neq \emptyset$.
Furthermore, the relation $f(\lambda x)=\lambda^nf(x)$ for all $x\in\mathbb{R}$
and all $\lambda \in\mathbb{Q}$ implies that $G(f)\cap \varphi_{\lambda}( K^*)=\emptyset $
for all rational number $\lambda\neq 0$, where $\varphi_{\lambda}(x,y)=(\lambda x,\lambda^ny)$.

Let us prove that the connected component of $G(f)$
which contains the point $(x,f(x))$ with $x\in ]x_0,x_1[$,
is the set $\{(x,f(x))\}$.
To prove this, we note that the sets $U_K^*,V_K^*$ separate any of these points
from the points $(y,f(y))$ of the graph satisfying $y\not\in ]x_0,x_1[$.
Thus it is only necessary to consider, to prove our claim, the following two cases:

\noindent \textbf{Case 1: $(x,f(x)) \in U_K^*$.}
The density of $G(f)$ in $\Gamma_f$ implies  there exist
an infinite sequence of open intervals $]a_n,b_n[ \subset ]x_0,x_1[$
such that $a_n<x<b_n$, $\lim_{n\to\infty}|a_n-b_n|=0$,
$(a_n,f(a_n)),(b_n,f(b_n)) \in V_{K^*}$.

\begin{figure}[htb]
\centering
\includegraphics[scale=0.3]{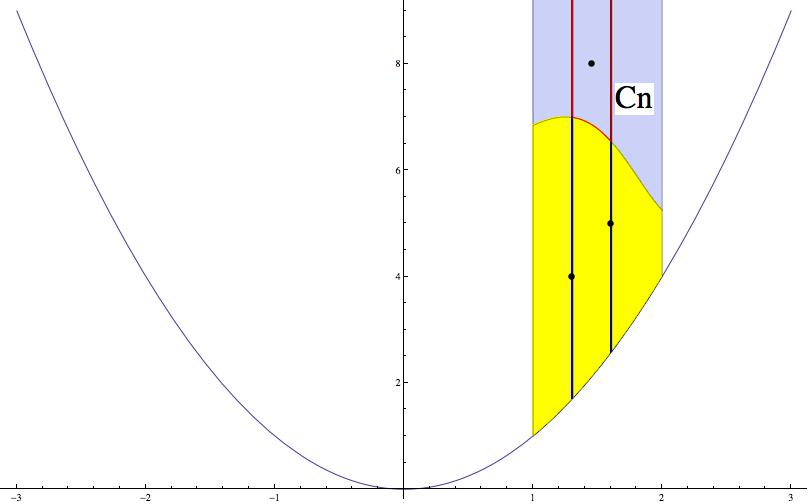}
\caption[]{A visualization of the sets $C_n$}
\end{figure}

Hence
$$ C_n=((\{a_n,b_n\}\times\mathbb{R})\cap U_{K^*})\cup (K^* \cap\overline{U_{K^*}})$$
is a sequence of connected subsets of the plane
which separates the point $(x,f(x))$ from any other point
$(y,f(y))$ with $y\neq x$, $y\in ]x_0,x_1[$,
and $G(f)\cap C_n=\emptyset$ for all $n$ (see the Figure).
It follows that $\{(x,f(x))\}$ is the connected component which contains the point $(x,f(x))$.

\noindent \textbf{Case 2: $(x,f(x)) \in V_K^*$.}
This case has an analogous proof to Case 1.

The proof ends now easily.
Indeed, if $(x,f(x))$ is any point of $G(f)$,
there exists $\lambda\in\mathbb{Q}$ such that
$x\in ]\lambda x_0,\lambda^nx_1[$ (since $\mathbb{Q}$ is a dense subset of $\mathbb{R}$)
and we can use the arguments above with $\varphi_{\lambda}(K^*)$ instead of $K^*$.


\end{remark}

\bibliographystyle{amsplain}


\bigskip

\footnotesize{Jose Maria Almira

Departamento de Matem\'{a}ticas, Universidad de Ja\'{e}n, Spain

E.P.S. Linares,  C/Alfonso X el Sabio, 28

23700 Linares (Ja\'{e}n) Spain

e-mail address: jmalmira@ujaen.es }



\vspace{1cm}

\footnotesize{Zolt\'{a}n~Boros

Institute of Mathematics,
University of Debrecen

P. O. Box: 12.

H--4010 Debrecen,
Hungary

e-mail address: zboros@science.unideb.hu }



\end{document}